\documentclass{article}
\usepackage[ansinew]{inputenc} 
\usepackage{graphicx}
\usepackage{amsmath,amsthm,amssymb,amscd}
\usepackage{color}
\usepackage[colorlinks]{hyperref}

\title {Combinatorial properties of irreducible\\ Laguerre polynomials in two variables}

\author{Nikolai A. Krylov \\ ~ \\
Siena College, Department of Mathematics\\
515 Loudon Road, Loudonville NY 12211, USA\\ ~ \\
nkrylov@siena.edu}

\date {}

\begin{document}

\newtheorem{thm}{Theorem}
\newtheorem{lem}{Lemma}
\newtheorem{claim}{Claim}
\newtheorem{dfn}{Definition}
\newtheorem{cor}{Corollary}
\newtheorem{prop}{Proposition}
\newtheorem{example}{Example}

\def\natu 		{\mathbb N}
\def\inte 		{\mathbb Z}
\def\rati 		{\mathbb Q}
\def\real		{\mathbb R}
\def\GCD 		{{\rm gcd}}

\def\lla 		{\longleftarrow}
\def\lra 		{\longrightarrow}
\def\ra 		{\rightarrow}
\def\hra 		{\hookrightarrow}
\def\lmt 		{\longmapsto}

\maketitle

\parskip=3mm

\begin{abstract}
Following our earlier work, where doubly indexed and irreducible over $\rati$ two-variable Laguerre polynomials 
were introduced, we prove for such polynomials some recurrence formulas and obtain a generating function. In addition, 
we show how certain sums of such polynomials with a fixed total degree relate to some standard polynomials.
\end{abstract}


\section{Introduction}

The generalized Laguerre polynomials in one variable can be defined for an arbitrary integer $n\geq 0$ and a non-negative 
integer parameter $\alpha $ by Rodrigues' relation (see, for example, \cite{Dunkl}, \S 1.4.2 or \cite{Szego}, \S 5.1, pages 101-103)  
$$
L_n^{(\alpha)} (x) = \frac{1}{n!} e^xx^{-\alpha} D^n\left(e^{-x}x^{n+\alpha}\right),~~~~\mbox{where}~~~~D^n:=\frac{d^n}{dx^n}.
$$
Expanding the definition of $L_n^{(\alpha)}(x)$ using the $n$-fold product rule and the Pochhammer symbol defined 
for all $x\in\real$ by
$$
(x)_0 = 1,~~~(x)_n = \prod\limits_{i=1}^n (x+i-1)~~~~\mbox{for all}~~n\in\natu,
$$
one immediately comes to the following explicit formulas:

\begin{equation}
\label{LaguerreD}
L_n^{(\alpha)} (x)  = \sum\limits_{j=0}^n\frac{(\alpha +1)_n\cdot (-n)_j}{(\alpha+1)_j \cdot n!}\cdot \frac{x^j}{j!} = 
\sum\limits_{j=0}^n\frac{(-1)^j}{j!} {n+\alpha \choose n - j} x^j
\end{equation}

\noindent Furthermore, for all $n\in \natu$ and $\alpha\in \natu\cup\{0\}$, such polynomials satisfy two well-known 
recurrence relations with $L_0^{(\alpha)}(x) = 1, ~ L_1^{(\alpha)}(x) = \alpha + 1 -x$:

\begin{equation}
\label{recurrence1}
(n+1)L_{n+1}^{(\alpha)} (x) = (2n+1-x+\alpha)L_n^{(\alpha)}(x) - (n+\alpha)L_{n-1}^{(\alpha)}(x),
\end{equation}

\noindent and

\begin{equation}
\label{recurrence2}
L_n^{(\alpha)} (x) = L_n^{(\alpha+1)}(x) - L_{n-1}^{(\alpha+1 )}(x).
\end{equation}

Sequence $\{L_n^{(\alpha)}(x)\}_{n=0}^{\infty}$ also has the following generating function (see \cite{Szego}, \S 5 for (1) - (4)).

\begin{equation}
\label{GenF1}
\sum\limits_{n=0}^{\infty} L_n^{(\alpha)}(x) t^n = \frac{e^{\frac{-xt}{1 - t}}}{(1-t)^{\alpha+1}}, ~~~ |t|<1.
\end{equation}

Another key property of $L_n^{(\alpha)}(x)~\forall n\in\natu$, is irreducibility over the rationals for certain values of $\alpha$ 
(\cite{Filaseta1}, \cite{Filaseta2}, \cite{Filaseta3}, \cite{Schur}, \cite{Sell}). If $\alpha \geq 2$, there are only finitely many examples of 
reducible Laguerre polynomials, e.g. $L_2^{(2)}(x) = 1/2(x-2)(x-6)$ (see \cite{Filaseta3} for details). Moreover, for a fixed 
$\alpha$, Laguerre polynomials $L_n^{(\alpha)}(x)$ are orthogonal over the interval $(0,\infty)$ with 
respect to the weight function $\omega(x) = e^{-x}x^{\alpha}$. That is 

\begin{equation}
\label{orthogonal}
\int\limits_{0}^{\infty} e^{-x}x^{\alpha} L_n^{(\alpha)}(x) L_m^{(\alpha)}(x) = \Gamma(\alpha+1)\binom{n+\alpha}{n}\delta_{m,n},
\end{equation}
where $m,n\in\natu\cup\{0\}$, and $\delta_{m,n}$ being the Kronecker delta  (see chapter 1 of \cite{Dunkl} or \S 5.1 of \cite{Szego}).

There are numerous examples of families of orthogonal polynomials in two or more variables, and certain properties of 
some multivariate Laguerre polynomials have been studied in \cite{Aktas}, \cite{Dattoli}, \cite{Dunkl}, and \cite{Suetin}. 
Polynomials considered in \cite{Aktas}, \cite{Dunkl} and \cite{Suetin} are defined as products of generalized Laguerre 
polynomials in single variables, and therefore reducible. Polynomials considered in \cite{Dattoli} are homogeneous and 
related to the classical Laguerre polynomials $L_n(x)$ via the identity ${\cal L}_n(x, y) =y^n\cdot L_n(x/y)$.

Instead of taking the product, one could use Rodrigues' relation with respect to the partial derivatives to introduce 
multivariable Laguerre polynomials. Following such approach, Krylov and Li obtained in \cite{Krylov}  multivariable 
Laguerre polynomials $L_{n_1,\ldots,n_r}(x_1,\ldots, x_r)$ irreducible over $\rati$. In particular, they proved that such polynomials 
in two variables satisfy a congruence property analogous to the one obtained by Carlitz for the Laguerre polynomials in 
one variable (see \cite{Krylov}, \S3,4 for details).

In this paper, like in \cite{Krylov}, we will concentrate on the two-variable case of $L_{n,m}(x,y)$. 
First we will prove three recurrence formulas for such polynomials, and then use one of the formulas to establish 
the following generating function for the sequence $\{L_{n,m}(x,y)\}_{n,m=0}^{\infty}$.

\begin{equation}
\label{GenF2}
\sum\limits_{n,m=0}^{\infty} L_{n,m}(x,y) s^n t^m = \frac{e^{\frac{-sx - ty}{1 - s - t}}}{1 - s - t}, ~~~~~ |s|+|t| < 1
\end{equation}

In the last section we use the above generating function to derive a few interesting formulas, which relate 
$L_{n,m}(x,y)$ to some classical polynomials. At the end we discuss the orthogonality of the sequence 
$\{L_n((x+y)/2)\}_{n=0}^{\infty}$.



\section{Recurrence formulas and generating function for $L_{n,m}(x,y)$}

Originally polynomials $L_{n,m}(x,y)$ have been introduced in \S2 of \cite{Krylov} via Rodrigues' relation. It was also 
shown in Theorem 1. of \cite{Krylov} that such polynomials can be given by the following explicit formulas for all integer 
$n,m\geq 0$. If one of the indices $n$ or $m$ is negative, we'll assume here that the polynomial is identically zero.

$$
L_{n,m}(x,y)  = \sum\limits_{i=0}^m\frac{(-1)^i}{i!} \cdot {m+n\choose m - i} \cdot L_n^{(i)}(x)\cdot y^i = 
\sum\limits_{s=0}^n\frac{(-1)^s}{s!}\cdot{n+m\choose n - s}\cdot L_m^{(s)}(y)\cdot x^s,~~~~~
\mbox{and}
$$
\begin{equation}
\label{LagXY}
L_{n,m}(x,y)  = \sum\limits_{i=0}^m\sum\limits_{s=0}^n\frac{(-1)^{i+s}}{i!\cdot s!}\cdot{m+n\choose m - i}
\cdot{n+i\choose n - s}\cdot x^s \cdot y^i
\end{equation}

Note: If we assume for a moment that the i-th power $(L_n(x)y)^i$ actually stands for the product 
of the Laguerre polynomial $L_n^{(i)}(x)$ with $y^i$ i.e. $(L_n(x)y)^i \equiv L_n^{(i)}(x)\cdot y^i$, then using (\ref{LaguerreD}), 
we could symbolically write our two-variable polynomials as compositions 
$$
L_{n,m}(x,y) = L_m^{(n)}\left(L_n(x) y\right) = L_n^{(m)}\left(L_m(y) x\right) = L_{m,n}(y,x).
$$

Classical Laguerre polynomial in one variable $L_n(x)$ can also be written as a hypergeometric series 
$$
L_n(x) = {}_1F_1(-n;1;x) =\sum\limits_{s=0}^{\infty} \frac{(-n)_s}{(1)_s}\frac{x^s}{s!},
$$
where we used the Pochhammer's symbol $(a)_s = a(a+1)\cdot\ldots\cdot(a-s+1)$ (see \S 5.2 of \cite{Szego}). 
Since 
$$
(-n)_s = (-n)(-n+1) \ldots (-n+s-1) = \frac{(-1)^s n!}{(n-s)!} ~ \mbox{and} ~ (-m)_i = \frac{(-1)^i m!}{(m-i)!},
$$
we can write the summand from the last equality in (\ref{LagXY}) as 
$$
\frac{(-1)^{i+s}}{i!\cdot s!}\cdot{m+n\choose m - i}
\cdot{n+i\choose n - s}\cdot x^s \cdot y^i = {n + m \choose n} \cdot \frac{(-n)_s (-m)_i}{(s+i)!} \cdot \frac{x^s}{s!}\frac{y^i}{i!}.
$$
Hence we can write our polynomial $L_{n,m}(x,y)$ as
$$
L_{n,m}(x,y) = \sum\limits_{s=0}^n\sum\limits_{i=0}^m {n +m \choose n} \cdot \frac{(-n)_s (-m)_i}{(s+i)!} 
\cdot \frac{x^s}{s!}\frac{y^i}{i!} = {n +m \choose n} \cdot \sum\limits_{s,i=0}^{\infty}  \frac{(-n)_s (-m)_i}{(1)_{s+i}}\frac{x^s}{s!}\frac{y^i}{i!},
$$
which is a presentation of $L_{n,m}(x,y)$ as one of the so-called Horn {\sl two-variable hypergeometric} series 
(see formula (21) on page 225 of \cite{Bateman}, \S 5.7.1.). Thus we also have the formula
\begin{equation}
\label{Horn}
L_{n,m}(x,y) = {n +m \choose n}\cdot  \Phi_2(-n, -m; 1; x, y).
\end{equation}

Now let us discuss the recurrence relations for $L_{n,m}(x,y)$.

\begin{thm}
The following two 6-term recurrence formulas hold for all real $x,y$, and all nonnegative integers $n$ and $m$.
\begin{equation}
\label{Recurr2}
\begin{split}
(n+1)L_{n+1,m+1}(x,y) = 2(n+1)L_{n+1,m}(x,y) + (2n+1 - x)L_{n,m+1}(x,y)\\ - (n+1)L_{n+1,m-1}(x,y)
- (2n+1 - x + y)L_{n,m}(x,y) - nL_{n-1,m+1}(x,y),
\end{split}
\end{equation}
\begin{equation}
\label{Recurr3}
\begin{split}
(m+1)L_{n+1,m+1}(x,y) =  (2m +1 - y)L_{n+1,m}(x,y) + 2(m+1)L_{n,m+1}(x,y) \\ - mL_{n+1,m-1}(x,y)
- (2m +1 + x - y)L_{n,m}(x,y) - (m+1)L_{n-1,m+1}(x,y),
\end{split}
\end{equation}
\begin{gather*}
L_{n,m}(x,y) = 0 ~ \mbox{if $nm < 0$}, ~  L_{0,0}(x,y)=1, ~ L_{1,0}(x,y) = L_1(x) = 1 - x, ~ L_{0,1}(x,y) = L_1(y) = 1 - y .
\end{gather*}
\end{thm}
\begin{proof}
Using formulas (\ref{LagXY}), one readily obtains the identites for $L_{n,m}(x,y)$ with $0\leq n+m\leq 1$. 
To prove the relation (\ref{Recurr2}) we rewire it as 
\begin{equation}
\label{Star1}
\begin{split}
(n+1)L_{n+1,m+1}(x,y) - (2n+1 - x)L_{n,m+1}(x,y) + y\cdot L_{n,m}(x,y) + nL_{n-1,m+1}(x,y) \\
 = 2(n+1)L_{n+1,m}(x,y) - (2n+1 - x)L_{n,m}(x,y) - (n+1)L_{n+1,m-1}(x,y)
\end{split}
\end{equation}
and compare the coefficients of $y^i$ on the R.H.S. and L.H.S. of (\ref{Star1}). We will also need the following two 
identities, which are trivial to check for all $i\in \{0,\ldots, m+1\}$ (the second identity can be also proved 
by applying the first one twice).
\begin{equation}
\label{Binom}
\begin{split}
& {m + n + 1\choose m + 1 - i} = {m+n\choose m + 1 - i}\cdot\biggl(1 + \frac{m+1-i}{n+i}\biggr), ~~~ \mbox{and} \\
& {m + n + 2\choose m + 1 - i} = {m+n\choose m + 1 - i}\cdot\biggl(1 + \frac{(m+1-i)(2n+m+2+i)}{(n+i)(n+1+i)}\biggr)
\end{split}
\end{equation}
Using the first of three equalities for $L_{n,m}(x,y)$ in (\ref{LagXY}) we can write the L.H.S. of (\ref{Star1}) as follows.
\begin{gather*}
\sum\limits_{i=0}^{m+1}(n+1)\frac{(-1)^i}{i!} {m+n+2 \choose m + 1 - i} L_{n+1}^{(i)}(x)y^i - 
\sum\limits_{i=0}^{m+1}(2n+1-x)\frac{(-1)^i}{i!} {m+n+1 \choose m +1 - i} L_{n}^{(i)}(x)y^i \\ 
+ \sum\limits_{j=0}^{m}\frac{(-1)^j}{j!} {m+n \choose m - j} L_{n}^{(j)}(x)y^{j+1} + 
\sum\limits_{i=0}^{m+1}n\frac{(-1)^i}{i!} {m+n \choose m + 1 - i} L_{n-1}^{(i)}(x)y^i \\
= \sum\limits_{i=0}^{m+1}\frac{(-1)^i}{i!} y^i \cdot \biggl( (n+1){m+n+2 \choose m + 1 - i} L_{n+1}^{(i)}(x) - 
(2n+1-x){m+n+1 \choose m +1 - i} L_{n}^{(i)}(x)\\ 
- i {m+n \choose m + 1 - i} L_{n}^{(i-1)}(x) + n {m+n \choose m + 1 - i} L_{n-1}^{(i)}(x) \biggr)
\end{gather*}
Using (\ref{Binom}) and (\ref{recurrence1}) with $\alpha$ replaced by $i$, we rewrite this last sum as
\begin{gather*}
= \sum\limits_{i=0}^{m+1}\frac{(-1)^i}{i!} y^i {m+n\choose m + 1 - i} \biggl( iL_n^{(i)}(x) - iL_{n-1}^{(i)}(x) + 
\frac{(n+1)(m+1-i)(m+2n+2+i)}{(n+i)(n+1+i)}L_{n+1}^{(i)}(x) \\ 
- (2n+1-x)\frac{m+1-i}{n+i}L_n^{(i)}(x) - iL_n^{(i-1)}(x) \biggr),
\end{gather*}
which simplifies to 
\begin{equation}
\label{LHS1}
\sum\limits_{i=0}^{m}\frac{(-1)^i}{i!} y^i {m+n\choose m - i} \biggl(
\frac{(n+1)(m+2n+2+i)}{(n+1+i)}L_{n+1}^{(i)}(x) - (2n+1-x)L_n^{(i)}(x) \biggr),
\end{equation}
since 
$$
iL_n^{(i)}(x) - iL_{n-1}^{(i)}(x) = iL_n^{(i-1)}(x)~ \mbox{by (\ref{recurrence2}), and also} ~ 
{m+n\choose m +1 - i}\frac{m+1-i}{n+i} = {m+n\choose m - i}.
$$ 
Notice also that the sum in (\ref{LHS1}) runs now until $m$, since $(m+1-i)/(n+i) = 0$ when $i = m+1$. Further, if we 
apply a similar approach to the R.H.S. of (\ref{Star1}), we obtain

\begin{gather*}
\sum\limits_{i=0}^{m} 2(n+1)\frac{(-1)^i}{i!} {m+n+1 \choose m - i} L_{n+1}^{(i)}(x)y^i - 
\sum\limits_{i=0}^{m}(2n+1-x)\frac{(-1)^i}{i!} {m+n \choose m - i} L_{n}^{(i)}(x)y^i \\ 
- \sum\limits_{i=0}^{m-1}(n+1)\frac{(-1)^i}{i!} {m+n \choose m - 1 - i} L_{n+1}^{(i)}(x)y^i \\
= \sum\limits_{i=0}^{m-1}\frac{(-1)^i}{i!} y^i \left( \biggl( 2{m+n+1 \choose m - i} 
- {m+n \choose m - 1 - i} \biggr) (n+1)L_{n+1}^{(i)}(x)
- (2n+1-x){m+n \choose m - i}L_n^{(i)}(x)\right) \\ 
+ y^m\frac{(-1)^m}{m!}\left(2(n+1)L_{n+1}^{(m)}(x) - 
(2n+1-x)L_n^{(m)}(x)\right),
\end{gather*}
which simplifies to 
\begin{equation}
\label{RHS1}
\sum\limits_{i=0}^{m}y^i\frac{(-1)^i}{i!} {m+n \choose m - i} \biggl( \frac{(n+1)(2n+m+2+i)}{n+1+i} L_{n+1}^{(i)}(x)
- (2n+1-x) L_n^{(i)}(x)\biggr),
\end{equation}
because for all $i\in\{0,\ldots m\}$ we have
$$
2{m+n+1 \choose m - i} - {m+n \choose m - 1 - i} = {m+n \choose m - i}\frac{2n+m+2+i}{n+1+i}, ~ \mbox{and also} ~ 
\frac{2n+m+2+m}{n+1+m} =2.
$$
Since (\ref{LHS1}) and (\ref{RHS1}) are identical, the proof of the 6-term recurrence formula (\ref{Recurr2}) is finished. 
In a similar way, or using a symmetry argument, one can show the analogous formula (\ref{Recurr3}), 
where $(n+1)$ is replaced by $(m+1)$.
\end{proof}

Now, multiplying all terms in (\ref{Recurr2}) by $(m+1)$, all terms in (\ref{Recurr3}) by $(n+1)$, subtracting both sides 
of the corresponding results and moving two polynomials with the total degree n+m+1 to one side while 
keeping the rest with the total degree n+m on the other, gives

\begin{cor}
$L_{n,m}(x,y)$ satisfy the following 5-term recurrence relation $\forall x,y\in\real$, and $\forall n,m\in\natu$.
\begin{equation}
\label{Recurr4}
\begin{split}
(n+1)(y+1)L_{n+1,m}(x,y) - (m+1)(x+1)L_{n,m+1}(x,y) = (n+1)L_{n+1,m-1}(x,y)\\ 
- (m+1)L_{n-1,m+1}(x,y)  + \biggl((m+1)(y-1-x) - (n+1)(x-1-y)\biggr)L_{n,m}(x,y)
\end{split}
\end{equation}
\end{cor}

Our next formula is a key step in proving (\ref{GenF2}) therefore it is stated separately.

\begin{lem}
The following identity holds for all $x,y\in\real$ and $\forall n\in\natu\cup\{0\}$.
\begin{equation}
\label{GFstep1}
\sum\limits_{m=0}^{\infty} n!L_{n,m}(x,y)t^m = \frac{e^{\frac{-ty}{1-t}}}{(1-t)^{n+1}} n! L_n\left(x + \frac{ty}{1-t}\right), ~~ |t|<1.
\end{equation}
\end{lem}
\begin{proof}
We will use strong induction on $n$ and consider first $n=0$. Since $L_{0,m}(x,y) = L_m(y)$ and $L_0(x) = 1$, if we apply 
(\ref{GenF1}) with $\alpha=0$, we obtain when $|t|<1$ 
$$
 \frac{e^{\frac{-ty}{1-t}}}{1-t} = \sum\limits_{m=0}^{\infty} L_m(y)t^m = \sum\limits_{m=0}^{\infty} 0!L_{0,m}(x,y)t^m,
$$
which is equivalent (\ref{GFstep1}) with $n=0$. For the rest of this proof, when referring to the 
formulas (\ref{LaguerreD}) - (\ref{recurrence2}), I will assume $\alpha=0$ unless stated otherwise. 
Now assume that (\ref{GFstep1}) holds true for all natural numbers $k\leq n$. Multiplying both sides of (\ref{recurrence1}) 
by $n!$ and substituting for $x$ the term $x + \frac{ty}{1-t}$, we can write the right hand side of (\ref{GFstep1}) as

\begin{gather*}
\frac{e^{\frac{-ty}{1-t}}}{(1-t)^{n+2}} (n+1)! L_{n+1}\left(x + \frac{ty}{1-t}\right) \\
= \frac{e^{\frac{-ty}{1-t}}}{(1-t)^{n+2}}\Biggl(\left(2n+1-x-\frac{ty}{1-t}\right)n!L_n\left(x + \frac{ty}{1-t}\right) - 
n^2(n-1)!L_{n-1}\left( x + \frac{ty}{1-t}\right)\Biggr) \\
= \frac{e^{\frac{-ty}{1-t}}}{(1-t)^{n+1}}\Biggl(\frac{\left(2n+1-x-\frac{ty}{1-t}\right)n!}{1-t} L_n\left(x + \frac{ty}{1-t}\right) - 
\frac{n^2 (n-1)!}{1-t} L_{n-1}\left( x + \frac{ty}{1-t}\right)\Biggr) \\
= \frac{e^{\frac{-ty}{1-t}}}{(1-t)^{n+1}} n!  L_n\left(x + \frac{ty}{1-t}\right) \Biggl(\frac{2n+1-x}{1-t} -\frac{ty}{(1-t)^2}\Biggr) 
- \frac{e^{\frac{-ty}{1-t}}}{(1-t)^{n}} (n -1)!  L_{n-1}\left(x + \frac{ty}{1-t}\right) \frac{n^2}{(1-t)^2}.
\end{gather*}
Next apply the induction hypothesis and rewrite the last line as
$$
\sum\limits_{m=0}^{\infty} n!L_{n,m}(x,y)t^m \Biggl(\frac{2n+1-x}{1-t} -\frac{ty}{(1-t)^2}\Biggr) - 
\sum\limits_{m=0}^{\infty} (n-1)!L_{n-1,m}(x,y)t^m \frac{n^2}{(1-t)^2},
$$
which implies that to finish the proof it's enough to show that
\begin{equation}
\label{EQ1}
\begin{split}
& \sum\limits_{m=0}^{\infty} (n+1)!L_{n+1,m}(x,y)t^m(1-t)^2 \\
& = \sum\limits_{m=0}^{\infty} n!L_{n,m}(x,y)t^m \biggl(2n+1-x-t(2n+1+y-x)\biggr) - 
\sum\limits_{m=0}^{\infty} n^2(n-1)!L_{n-1,m}(x,y)t^m.
\end{split}
\end{equation}
To do it, let us again compare the coefficients of $t^k$ on both sides of (\ref{EQ1}) and use the first recurrence  
formula (\ref{Recurr2}). Multiplying out $(1-t)^2$ on the left-hand side of (\ref{EQ1}) gives
\begin{gather*}
\sum\limits_{m=0}^{\infty} (n+1)!L_{n+1,m}(x,y)t^m - \sum\limits_{m=0}^{\infty} 2(n+1)!L_{n+1,m}(x,y)t^{m+1} +  
\sum\limits_{m=0}^{\infty} (n+1)!L_{n+1,m}(x,y)t^{m+2}\\
= (n+1)!L_{n+1,0}(x,y) + (n+1)!L_{n+1,1}(x,y)t -  2(n+1)!L_{n+1,0}(x,y) t \\
+ \sum\limits_{m=0}^{\infty} \biggl( (n+1)!L_{n+1,m+2}(x,y)t^{m+2} - 2 (n+1)!L_{n+1,m+1}(x,y)t^{m+2} + 
(n+1)!L_{n+1,m}(x,y)t^{m+2}\biggr),
\end{gather*}
which equals
\begin{equation}
\label{EQ2}
\begin{split}
& (n+1)!L_{n+1}(x) + \biggl( (n+1)!L_{n+1,1}(x,y) - 2(n+1)!L_{n+1}(x)\biggr)t \\
& + \sum\limits_{m=0}^{\infty} (n+1)!\biggl(L_{n+1,m+2}(x,y) - 2 L_{n+1,m+1}(x,y) + L_{n+1,m}(x,y) \biggr)t^{m+2}.
\end{split}
\end{equation}
On the right-hand side of (\ref{EQ1}) we correspondingly obtain
\begin{equation}
\label{EQ3}
\begin{split}
& n!\biggl(L_{n,0}(x,y)(2n+1-x) - nL_{n-1,0}(x,y)\biggr) \\ 
& + n!\biggl(L_{n,1}(x,y)(2n+1-x) - L_{n,0}(x,y)(2n+1+ y -x) - nL_{n-1,1}(x,y)\biggr)t \\
& + \sum\limits_{m=0}^{\infty} n! \biggl(L_{n,m+2}(x,y)(2n + 1 - x) - L_{n,m+1}(x,y)(2n + 1 + y - x) 
-  nL_{n-1,m+2}(x,y)\biggr)t^{m+2}.
\end{split}
\end{equation}
Equality of free terms of (\ref{EQ2}) and (\ref{EQ3}) follows from (\ref{recurrence1}), and equality of the coefficients of $t^k$ 
for $k\geq 2$ follows exactly from our formula (\ref{Recurr2}) with $m+1$ substituted for $m$. As for the coefficients of $t$, one could 
prove their equality directly from the defining formulas (\ref{LagXY}). However, since $L_{n+1,-1}(x,y)$ is identically zero, 
assuming $m = 0$ in (\ref{Recurr2}) we also have 
\begin{gather*}
(n+1)L_{n+1,1}(x,y) = 2(n+1)L_{n+1,0}(x,y)\\ + (2n+1 - x)L_{n,1}(x,y) - (2n+1 - x + y)L_{n,0}(x,y) - nL_{n-1,1}(x,y),
\end{gather*}
which implies the required equality for the coefficients of $t$ and finishes our proof of (\ref{GFstep1}).
\end{proof}


Let us now prove formula (\ref{GenF2}).

\begin{thm}
Doubly indexed sequence of Laguerre polynomials $L_{m,n}(x,y)$ with $n,m\geq 0$ has the following generating function.
\begin{equation}
\label{genfun2}
\sum\limits_{n,m=0}^{\infty} L_{n,m}(x,y) s^n t^m = \frac{e^{\frac{-sx - ty}{1 - t - s}}}{1 - t - s}, ~~~~~ |s| + |t| < 1
\end{equation}
\end{thm}

\begin{proof}
The proof is based on the previous lemma and the formulas (5.1.13) and (5.1.14) from \cite{Szego}, \S 5.1. The 
last one states 
$$
\frac{d}{dx}L_n^{(\alpha)}(x) = - L_{n-1}^{(\alpha+1)}(x) = x^{-1}\biggl(nL_n^{(\alpha)}(x) - (n+\alpha)L_{n-1}^{(\alpha)}(x)\biggr),
$$
which in particular, implies
\begin{equation}
\label{Eq13}
(n+1)L_{n+1}(x) = (n+1)L_n(x) - xL_n^{(1)}(x).
\end{equation}
Assuming for a moment 
\begin{equation}
\label{DER}
\frac{\partial ^n}{\partial s^n}\left(\frac{e^{\frac{-xs-yt}{1-s-t}}}{1-s-t}\right) = \frac{e^{\frac{-xs-yt}{1-s-t}}}{(1-s-t)^{n+1}}n!L_n\left(\frac{(1-t)x + ty}{1-s-t}\right),
\end{equation}
substituting 0 for $s$ and applying our Lemma 1. will give 
$$
\left.\frac{\partial ^n}{\partial s^n}\right|_{s=0} \left(\frac{e^{\frac{-xs-yt}{1-s-t}}}{1-s-t}\right)
= \frac{e^{\frac{-yt}{1-t}}}{(1-t)^{n+1}}n!L_n\left(x + \frac{ty}{1-t}\right) = \sum\limits_{m=0}^{\infty} n!L_{n,m}(x,y)t^m.
$$
Then differentiating $m$ times with respect to $t$ and evaluating at $t=0$ will prove the result
$$
\left.\frac{\partial ^{n+m}}{\partial s^n\partial t^m}\right|_{s=0, t=0} \left(\frac{e^{\frac{-xs-yt}{1-s-t}}}{1-s-t}\right)
= n!m! L_{n,m}(x,y).
$$
Thus to finish the proof, we need to establish (\ref{DER}), which we will do by induction on $n$. The base of induction 
when $n=0$ is trivial, so let's assume that (\ref{DER}) holds true and show that the formula remains true when we 
differentiate $n+1$ times. Using the induction hypothesis 
together with the product rule and the first equality from (5.1.14), \S 5.1 of  \cite{Szego} we obtain
\begin{gather*}
\frac{\partial ^{n+1}}{\partial s^{n+1}}\left(\frac{e^{\frac{-xs-yt}{1-s-t}}}{1-s-t}\right) = -n!L_{n-1}^{(1)}\left(\frac{(1-t)x+ty}{1-s-t}\right)\cdot 
\frac{(1-t)x+ty}{(1-s-t)^2}\cdot \frac{e^{\frac{-xs-yt}{1-s-t}}}{(1-s-t)^{n+1}} \\
+ n!L_n\left(\frac{(1-t)x + ty}{1-s-t}\right)\cdot\left(\frac{e^{\frac{-xs-yt}{1-s-t}}(n+1)}{(1-s-t)^{n+2}} + 
\frac{e^{\frac{-xs-yt}{1-s-t}}}{(1-s-t)^{n+1}} \frac{xt-x-yt}{(1-s-t)^2} \right)\\
= \frac{e^{\frac{-xs-yt}{1-s-t}}(n+1)!}{(1-s-t)^{n+2}} L_n\left(\frac{(1-t)x+ty}{1-s-t}\right) \\ 
- \frac{e^{\frac{-xs-yt}{1-s-t}}(n+1)!}{(1-s-t)^{n+2}}\left[ 
\frac{(1-t)x+ty}{(n+1)(1-s-t)}\left(L_{n-1}^{(1)}\left(\frac{(1-t)x+ty}{1-s-t}\right) + L_n\left(\frac{(1-t)x+ty}{1-s-t}\right)  \right)\right].
\end{gather*}
Applying $L_{n-1}^{(1)}(x) + L_n(x) = L_n^{(1)}(x)$ (see \cite{Szego}, (5.1.13)), we can write the last part of the formula as
\begin{gather*}
\frac{e^{\frac{-xs-yt}{1-s-t}}n!}{(1-s-t)^{n+2}}\biggl((n+1)L_n\left(\frac{(1-t)x+ty}{1-s-t}\right) - 
\frac{(1-t)x+ty}{1-s-t} L_n^{(1)}\left(\frac{(1-t)x+ty}{1-s-t}\right)\biggr),
\end{gather*}
which via (\ref{Eq13}), also equals 
$$
\frac{e^{\frac{-xs-yt}{1-s-t}}(n+1)!}{(1-s-t)^{n+2}}L_{n+1}\left(\frac{(1-t)x+ty}{1-s-t}\right).
$$ 
This finishes our induction proof of (\ref{DER}).
\end{proof}


\section{Other properties of $L_{n,m}(x,y)$}

For each positive integer $k$ there are exactly $k+1$ polynomials $L_{n,m}(x,y)$ of the total degree $k = n+m$. 
If we add all such polynomials together, we will obtain a constant multiple of the classical 
Laguerre polynomial $L_k((x+y)/2)$. The corresponding alternating sum will give us a constant multiple of $(x-y)^k$. 
Here are precise statements, where we assume $x\neq 0$ and $y\neq 0$, in (\ref{xSum}) and (\ref{xySum}).

\begin{thm}
For every fixed $k\in\natu$, we have the identities
\begin{gather}
\label{Sum}
\sum\limits_{n+m=k} L_{n,m}(x,y) = 2^kL_k\biggl(\frac{x+y}{2}\biggr),\\
\label{ASum}
\sum\limits_{n+m=k} (-1)^m L_{n,m}(x,y) = \frac{1}{k!} (y-x)^k,\\
\label{xSum}
\sum\limits_{n+m=k} x^m L_{n,m}(x,1/x) =  L_k(1)(1+x)^k,\\
\label{xySum}
\sum\limits_{n+m=k} x^n y^m L_{n,m}(1/x,-1/y) =  (x + y)^k
\end{gather}
\end{thm}
\begin{proof}
To prove (\ref{Sum}), we will use the generating functions for $L_k(x)$ and $L_{n,m}(x,y)$. Since 
$$
\sum\limits_{k=0}^{\infty} L_k(x)t^k = \frac{e^{\frac{-tx}{1-t}}}{1-t}, ~~~ \mbox{we have} ~~~ 
\sum\limits_{k=0}^{\infty} L_k\biggl(\frac{x+y}{2}\biggr)t^k = \frac{e^{\frac{-t(x+y)/2}{1-t}}}{1-t}, ~ |t|<1,
$$
which after the substitution $s=t/2$ becomes
$$
\sum\limits_{k=0}^{\infty} 2^kL_k\biggl(\frac{x+y}{2}\biggr)s^k = \frac{e^{\frac{-s(x+y)}{1-2s}}}{1-2s}, ~ |s|<1/2.
$$
On the other hand, we now have (\ref{genfun2}), and if we make there $s=t$ and let $m=k - n$, we obtain 
$$
\sum\limits_{k=0}^{\infty}\biggl(\sum\limits_{n=0}^k L_{n, k-n}(x,y)\biggr) s^k = 
\sum\limits_{n,m=0}^{\infty} L_{n,m}(x,y) \cdot s^{n+ m} = 
\frac{e^{\frac{-s(x +y)}{1 - 2s}}}{1 - 2s} = \sum\limits_{k=0}^{\infty} 2^kL_k\biggl(\frac{x+y}{2}\biggr)s^k,
$$
which give (\ref{Sum}). Similarly, if we make $t = -s$ in (\ref{genfun2}), we obtain the following.
$$
\sum\limits_{k=0}^{\infty}\biggl(\sum\limits_{n=0}^k (-1) ^{k-n}L_{n, k-n}(x,y)\biggr) s^k = 
\sum\limits_{n,m=0}^{\infty} (-1)^m L_{n,m}(x,y) \cdot s^{n+ m} = e^{s(y - x)}.
$$
Furthermore, using Taylor's expansion we have 
$$
e^{s(y - x)} = \sum\limits_{k=0}^{\infty} \frac{1}{k!}(s(y-x))^k =  \sum\limits_{k=0}^{\infty} \biggl(\frac{(y-x)^k}{k!}\biggr)s^k,
$$
which implies (\ref{ASum}). The last two identities are proved in the same manner. Indeed, first substituting 
$t = xs$ and $y = 1/x$ in the generating function (\ref{genfun2}) for $L_{n,m}(x,y)$ gives
\begin{equation}
\label{xSum2}
\sum\limits_{n,m=0}^{\infty} L_{n,m}(x,y)s^nt^m = \sum\limits_{k=0}^{\infty} 
\left(\sum\limits_{n=0}^k L_{n,k-n}\left(x, 1/x\right) \cdot x^{k-n} \right)s^k = \frac{e^{\frac{-sx-s}{1-s-sx}}}{1-s-sx}.
\end{equation}
Now, using the generating function (\ref{GenF1}) for $L_n(x)$ and the substitution $z = s + sx$, we obtain
$$
 \frac{e^{\frac{-sx-s}{1-s-sx}}}{1-s-sx} = \frac{e^\frac{-z}{1-z}}{1-z} = \sum\limits_{k=0}^{\infty} L_k(1) z^k = 
 \sum\limits_{k=0}^{\infty} \biggl(L_k(1) (x+1)^k\biggr)s^k, 
$$
which together with (\ref{xSum2}) proves (\ref{xSum}). To prove (\ref{xySum}) use the substitutions $s = zx$ and 
$t = zy$ in (\ref{genfun2}) with $L_{n,m}(1/x, -1/y)$. It will give the geometric series $1/(1 - (x+y)z)$, 
which is of course the generating function for $(x+y)^k$.
\end{proof}

Identity (\ref{Sum}) suggests various interesting properties for such sums of polynomials $L_{n,m}(x,y)$. One 
of such properties is the orthogonality relation. We are going to discuss it next, but first let's introduce a separate 
name for such sums.
\begin{dfn} For each $k\in\natu\cup\{0\}$ and all $x,y\in\real$, define
$$
LS_k(x,y):=\frac{1}{2^k}\cdot \sum\limits_{n+m=k} L_{n,m}(x,y) = L_k\biggl(\frac{x+y}{2}\biggr).
$$
\end{dfn}

Here is the corresponding orthogonality property for polynomials $LS_k(x,y)$. 

\begin{prop}
Polynomials $LS_k(x,y)$ are orthonormal over the open first quadrant in the plane 
$\{(x,y)\in\real^2 ~|~ x>0 ~\mbox{and} ~ y >0\}$ 
with respect to the weight function 
$$
\omega(x,y) = \frac{e^{\frac{-x-y}{2}}}{2\pi\sqrt{xy}}.
$$
In particular, for all $m,n\in\natu\cup\{0\}$ we have the following identity.
\begin{equation}
\int\limits_0^{\infty}\int\limits_0^{\infty} \frac{e^{\frac{-x-y}{2}}}{2\pi\sqrt{xy}}
\cdot LS_{n}(x,y)\cdot LS_{m}(x,y) ~dxdy = \delta_{n,m}
\end{equation}
\end{prop}
\begin{proof}
If we use the change of variables $u=(x+y)/2$, $v=x/y$, the corresponding transformation will be a diffeomorphism 
of the first quadrant onto itself, with the Jacobian
$$
\begin{vmatrix}
x_u & x_v \\
y_u & y_v\\
\end{vmatrix} = 
\begin{vmatrix}
\frac{2v}{v+1} & \frac{2u(v+1) - 2uv}{(v+1)^2} \\
\frac{2}{v+1} & \frac{-2u}{(v+1)^2}\\
\end{vmatrix} = \frac{4u}{(v+1)^2}.
$$
After this change our double integral can be written as the product
$$
\frac{1}{\pi} \int\limits_0^{\infty} \frac{dv}{\sqrt{v}(v+1)} \cdot \int\limits_0^{\infty} e^{-u} L_{n}(u)L_{m}(u) ~du = 
\int\limits_0^{\infty} e^{-u} L_{n}(u)L_{m}(u) ~du = \delta_{n,m},
$$
which gives the required result since the classical Laguerre polynomials $\{L_k(x)\}_{k=0}^{\infty}$ are orthonormal 
over $(0,\infty)$ with respect to the weight function $e^{-x}$ (see (\ref{orthogonal})).
\end{proof}



\end{document}